\documentclass[12pt]{amsart} 
\usepackage{amsfonts,latexsym}
\usepackage{amsmath}
\usepackage{amscd}
\usepackage{float,amsmath,amssymb,mathrsfs,bm,multirow,graphics}
\usepackage[dvips]{graphicx}
\usepackage[percent]{overpic}

\renewcommand{\theequation}{\mbox{\arabic{section}.\arabic{equation}}}

\newtheorem{theorem}{Theorem}[section]

\newtheorem{corollary}[theorem]{Corollary} 
\newtheorem{lemma}[theorem]{Lemma}

\newtheorem{remark}[theorem]{Remark}

\newcommand{\Rot}{\text{\upshape Rot}}
\newcommand{\Jac}{\text{\upshape Jac}}

\DeclareMathOperator{\diver}{div}

\input epsf
\title[Amari-Chentsov connections and their geodesics] 
{Amari-Chentsov connections and their geodesics on homogeneous 
 spaces of diffeomorphism groups}
\author{Jonatan Lenells} 
\address{J.L.: Department of Mathematics, Baylor University, Waco, TX 76798, USA}
\email{Jonatan\_Lenells@baylor.edu}
\author{Gerard Misio\l ek}
\address{G.M.: Department of Mathematics, University of Notre Dame, IN 46556, USA} 
\email{gmisiole@nd.edu} 
\thanks{G.M. was partially supported by the James D. Wolfensohn Fund and 
the Friends of the Institute for Advanced Study Fund. J.L. acknowledges support from the EPSRC, UK}

\begin{document}

\begin{abstract} 
\noindent 
We study the family of $\alpha$-connections of Amari-Chentsov on the homogeneous space 
$\mathcal{D}(M)/\mathcal{D}_\mu(M)$ of diffeomorphisms modulo volume-preserving diffeomorphims 
of a compact manifold $M$.
%equipped with a Sobolev $\dot{H}^1$ metric. 
We show that in some cases their geodesic equations yield completely integrable 
Hamiltonian systems. 
\end{abstract} 

\maketitle



\numberwithin{equation}{section} 

%\tableofcontents

%%%%%%%%%%%%%%%%%%%%
\section{Introduction} 
In a recent paper \cite{klmp} several key notions of geometric statistics were developed 
on diffeomorphism groups and their quotient spaces equipped with right-invariant metrics. 
For example, the metric defined by (the homogeneous part of) the Sobolev $H^1$ 
inner product of vector fields on the underlying compact manifold was shown to induce 
on the quotient of the diffeomorphism group by its subgroup of volume-preserving diffeomorphisms 
%(which, after suitable normalization, can be identified with the space of probability densities) 
an infinite-dimensional analogue of the Fisher-Rao metric 
while geodesics of its Levi-Civita connection - when the underlying manifold is the circle - 
were shown to be related to solutions of a well-known one-dimensional completely integrable 
equation \cite{km, le}. 
Furthermore, the authors also described analogues of the so-called $\alpha$-connections 
introduced in geometric statistics by Chentsov \cite{ch} and Amari \cite{AN} 
and pointed out integrability of another geodesic equation corresponding to $\alpha=-1$. 

Our goals in this paper are to provide a proof of the construction in \cite{klmp} 
of Amari-Chentsov connections for circle diffeomorphisms (Theorem \ref{thm1}), 
to generalize this construction to diffeomorphism groups of higher-dimensional manifolds (Theorem \ref{thm3}) 
and finally, as a by-product, to show integrability of the geodesic equations corresponding to $\alpha=1$ (Theorem \ref{affineremark}, Corollary \ref{affineremarkndim}). 

%%%%%%%%%%
\subsection{Diffeomorphism groups and the Fisher-Rao metric} 

Let $M$ be a compact Riemannian manifold without boundary. 
Let $\mathcal{D}(M)$ denote the group of smooth diffeomorphisms of $M$ and let $\mathcal{D}_\mu(M)$ 
be its subgroup of diffeomorphisms preserving the volume form $\mu$ on $M$. 
It is well-known that the completion of $\mathcal{D}(M)$ (resp. $\mathcal{D}_\mu(M)$) 
in the $H^s$ Sobolev norm with $s>n/2 +1$ can be equipped with the structure of 
a smooth Hilbert manifold whose tangent space at the identity diffeomorphism $e$ consists of 
all $H^s$ vector fields (resp. all divergence-free $H^s$ vector fields) on $M$. 
However, for what follows it will be sufficient to work with the smooth category. 
It will also be convenient to normalize the Riemannian volume $\mu(M)=1$. 
For any $\eta \in \mathcal{D}(M)$ and any $V, W \in T_\eta\mathcal{D}(M)$ 
we set 
\begin{align} \label{H1met} 
\langle V, W \rangle_{\dot{H}^1} 
= 
\frac{1}{4} \int_M \mathrm{div}\, v \, \mathrm{div}\, w \, d\mu 
\end{align} 
where $V=v\circ\eta$ and $W=w\circ\eta$ with $v, w \in T_e\mathcal{D}(M)$ 
to obtain a right-invariant (degenerate) $\dot{H}^1$ metric on $\mathcal{D}(M)$. 

The geometry of this metric turns out to be particularly remarkable. 
The homogeneous space $\mathcal{D}(M)/\mathcal{D}_\mu(M)$ 
can be naturally identified with the set of smooth probability densities, i.e., 
smooth functions $\rho >0$ on $M$ satisfying the condition $\int_M \rho \, d\mu = 1$. 
If $N$ is an $n$-dimensional submanifold of such densities $\rho = \rho_{t_1\dots t_n}$ 
parameterized by 
%a subset of the euclidean space 
$(t_1, \dots , t_n) \in \mathbb{R}^n$ 
then recall that the Fisher-Rao metric on $N$ is given by the formula 
\begin{align} \label{FRmet} 
g_{ij} = \int_M \frac{\partial\log\rho}{\partial t_i} \frac{\partial\log\rho}{\partial t_j} \rho\, d\mu 
\quad\qquad 
1 \leq i,j \leq n. 
\end{align} 
It turns out that 
the right-invariant $\dot{H}^1$ metric defined by \eqref{H1met} on $\mathcal{D}(M)$ 
descends to a (nondegenerate) metric on $\mathcal{D}(M)/\mathcal{D}_\mu(M)$ 
and 
the map $\eta \mapsto \sqrt{\mathrm{Jac}_\mu\eta}$ defines an isometry between 
$\mathcal{D}(M)/\mathcal{D}_\mu(M)$ and a subset of the unit sphere in $L^2(M,d\mu)$ 
with the canonical round metric. 
Furthermore, 
the restriction of the $\dot{H}^1$ metric to any finite-dimensional submanifold $N$ of 
$\mathcal{D}(M)/\mathcal{D}_\mu(M)$ coincides with the Fisher-Rao metric \eqref{FRmet} 
on $N$ 
while 
its Riemannian distance is the spherical Hellinger distance between probability densities on $M$. 
Proofs of these statements can be found in Section 3 of \cite{klmp}. 

Thus, in the framework of diffeomorphism groups, information geometry associated with 
the Fisher-Rao metric and its spherical Hellinger distance can be viewed as an $\dot{H}^1$ analogue 
of standard optimal transport associated with the metric\footnote{This metric is sometimes called Otto's metric.} 
on $\mathcal{D}(M)/\mathcal{D}_\mu(M)$ induced by the (non-invariant) $L^2$ metric 
on $\mathcal{D}(M)$ and whose Riemannian distance is the celebrated Kantorovich (or Wasserstein) 
distance, cf. \cite{v}.

%%%%%%%%%%%%%%%%%%%%%%%%%%%%%%%%%%%%%
\subsection{Divergence functions and dual connections} 

Recall that a divergence on an $n$-dimensional manifold $N$ is a smooth function 
$
D:  N \times N \to \mathbb{R}
$ 
satisfying 
$
D(p \| q) \geq 0
$ 
with equality if and only if $p=q$ and such that the matrix $g^D_{ij}$ defined in a chart at $p \in N$ by
\begin{equation} \label{eq:Dmetric} 
g^D_{ij}(p) 
= 
-\tfrac{\partial}{\partial p_i}\tfrac{\partial}{\partial q_j} D(p \| q)\vert_{p=q} 
\qquad\qquad 
1 \leq i,j \leq n 
\end{equation} 
is strictly positive definite for every $p \in N$. Equation (\ref{eq:Dmetric}) defines a Riemannian metric on $N$ with covariant derivative determined by 
\begin{equation} \label{eq:Dconn} 
\Gamma^D_{ij,k} 
= 
- \tfrac{\partial}{\partial p_i} \tfrac{\partial}{\partial p_j} \tfrac{\partial}{\partial q_k} D(p \| q)\vert_{p=q} 
\qquad\qquad 
1 \leq i,j,k \leq n. 
\end{equation} 

In what follows we shall consider on $\mathcal{D}(M)\times\mathcal{D}(M)$ the functions 
\begin{align} \label{Dalpha} 
&D^{(\alpha)}(\xi \| \eta) 
= 
\frac{1}{1 - \alpha^2} \left( 1 - \int_{M} 
( \mathrm{Jac}_\mu \xi )^{\frac{1-\alpha}{2}} ( \mathrm{Jac}_\mu \eta )^{\frac{1+\alpha}{2}} 
d\mu\right), 
\qquad 
-1 < \alpha < 1  
\\  \label{Dalpha1} 
&D^{(-1)}(\xi \| \eta) = D^{(1)}(\eta\|\xi) 
= 
\frac{1}{4} \int_M \big( \log\mathrm{Jac}_\mu\xi - \log\mathrm{Jac}_\mu\eta \big) \mathrm{Jac}_\mu\xi \, d\mu. 
\end{align}
These functions are well-defined on the homogeneous space $\mathcal{D}(M)/\mathcal{D}_\mu(M)$ 
and satisfy $D^{(\alpha)}(\xi\|\eta) \geq 0$ with equality if and only if $\xi$ and $\eta$ project onto 
the same probability density on $M$. 

Furthermore, recall that two affine connections $\nabla$ and $\nabla^*$ 
on a Riemannian manifold $N$ are called dual (or conjugate) relative to the Riemannian metric 
if for any vector fields $U$, $V$ and $W$ they satisfy 
$
W \langle U, V \rangle 
= 
\langle \nabla_W U, V \rangle + \langle U, \nabla^*_W V \rangle. 
$
Basic facts about dual connections and divergences can be found in \cite{AN} 
or \cite{kv}.

%%%%%%%%%%%%%%%
\section{The one-dimensional case: $\mathcal{D}(S^1)/\Rot(S^1)$} \label{section2} 

We first consider the case when the underlying manifold is the circle $S^1 = \mathbb{R}/\mathbb{Z}$. In this case $\mathcal{D}_\mu(S^1)$ is the space of rigid rotations $\Rot(S^1) \simeq S^1$. 
It will be convenient to identify the homogeneous space $\mathcal{D}(S^1)/\Rot(S^1)$ 
with the subgroup of circle diffeomorphisms which fix a prescribed point, for example with 
$\{ \eta \in \mathcal{D}(S^1): \eta(0)=0 \}$. Its tangent space at the identity can be identified with 
the space of smooth periodic functions vanishing at $0$. 
Furthermore, given any such function $u(x)$ we define the operator 
\begin{equation} \label{A} 
A^{-1}u(x) = - \int_0^x \int_0^y u(z) \, dz dy + x \int_0^1 \int_0^y u(z) \, dz dy 
\end{equation} 
i.e., the inverse of $A = -\partial_x^2$. 

%Consider the following functions on the product $\mathcal{D}(S^1)\times\mathcal{D}(S^1)$ 
%%
%\begin{align} \label{Dalpha} 
%&D^{(\alpha)}(\xi \| \eta) 
%= 
%\frac{1}{1 - \alpha^2} \left( 1 - \int_{S^1} \xi_x^{\frac{1-\alpha}{2}} \eta_x ^{\frac{1+\alpha}{2}} dx\right), 
%\qquad 
%-1 < \alpha < 1  
%\\  \label{Dalpha1} 
%&D^{(-1)}(\xi \| \eta) = D^{(1)}(\eta\|\xi) 
%= 
%\int_{S^1} \xi_x \log{\frac{\xi_x}{\eta_x}} \, dx. 
%\end{align}
%%
%$D^{(\alpha)}$ are well-defined on $\mathcal{D}(S^1)/\Rot(S^1)$ 
%and satisfy $D^{(\alpha)}(\xi\|\eta) \geq 0$. 

The following is a reformulation of Theorem 6.2 stated (without proof) in \cite{klmp}. 
Our first objective is to provide a proof of this result. 

\begin{theorem}[cf. \cite{klmp}, Section 6] \label{thm1} \mbox{} 
\begin{enumerate} 
\item[(i)] 
Each divergence $D^{(\alpha)}$ induces on $\mathcal{D}(S^1)/\Rot(S^1)$ 
the $\dot{H}^1$ metric and an affine connection $\nabla^{(\alpha)}$ whose Christoffel symbols 
are given by 
\begin{equation} \label{Gammaformula} 
\Gamma^{(\alpha)}_\eta (W, V) 
= 
-\frac{1+\alpha}{2} \Big\{ A^{-1}\partial_x \big[ (V\circ\eta^{-1})_x (W\circ\eta^{-1})_x \big]\Bigr\} \circ\eta, \quad -1 \leq \alpha \leq 1. 
\end{equation} 
\item[(ii)] 
For any $\alpha$ the connections $\nabla^{(\alpha)}$ and $\nabla^{(-\alpha)}$ are dual 
with respect to the $\dot{H}^1$ metric and $\nabla^{(0)}$ is the self-dual Levi-Civita connection. 
\item[(iii)] 
The geodesic equation of $\nabla^{(\alpha)}$ on $\mathcal{D}(S^1)/\Rot(S^1)$ 
is the generalized Proudman-Johnson equation 
\begin{equation} \label{PJ} 
u_{txx} + (2-\alpha)u_x u_{xx} + u u_{xxx} = 0. 
\end{equation} 
In particular, $\alpha =0$ yields the completely integrable Hunter-Saxton equation 
\begin{equation} \label{HS} 
u_{txx} + 2u_x u_{xx} + u u_{xxx} =0 
\end{equation} 
and $\alpha =-1$ yields the completely integrable $\mu$-Burgers equation 
\begin{equation} \label{muB} 
u_{txx} + 3u_x u_{xx} + u u_{xxx} = 0.
\end{equation} 
\end{enumerate} 
\end{theorem} 
\begin{remark} \upshape 
The equation corresponding to $\alpha = 1$ also turns out to be integrable and its solutions 
can be given explicitly, see Theorem \ref{affineremark} below.
\end{remark} 
\begin{proof} 
The metrics induced by $D^{(\alpha)}$ and their connections can be calculated 
essentially as in finite dimensions using formulas \eqref{eq:Dmetric} and \eqref{eq:Dconn}. 
We first assume that $\alpha \neq \pm 1$. 
Given any tangent vectors $V$ and $W$ at $\eta \in \mathcal{D}(S^1)$ let $\eta(s, t)$ 
be a two-parameter family of diffeomorphisms in $\mathcal{D}(S^1)$ such that 
$\eta(0,0) = \eta$ with $\frac{\partial}{\partial s} \eta(0,0) = V$ and $\frac{\partial}{\partial t} \eta (0,0) = W$. 
From (\ref{eq:Dmetric}) and \eqref{Dalpha} we have 
\begin{align} \nonumber 
\langle V, W \rangle_{\alpha} 
&= 
- \tfrac{\partial}{\partial s}\big|_{s = 0} \tfrac{\partial}{\partial t}\big|_{t= 0} 
D^{(\alpha)}(\eta(s,0) \| \eta(0,t) ) 
\\  \label{alphametric} 
&= 
\frac{1}{1 - \alpha^2} \, 
\tfrac{\partial}{\partial s}\big|_{s = 0} \tfrac{\partial}{\partial t}\big|_{t= 0} 
\int_{S^1} \eta_x(s,0)^{\frac{1-\alpha}{2}} \eta_x(0,t)^{\frac{1+\alpha}{2}} dx 
\\ \nonumber 
&= 
\frac{1}{4} \int_{S^1} V_x W_x \, \eta_x^{\frac{-1-\alpha}{2}} \eta_x^{\frac{-1+\alpha}{2}} dx 
= 
\frac{1}{4} \int_{S^1} \frac{V_x W_x}{\eta_x} \, dx  
\\ \nonumber 
&= 
\langle V, W \rangle_{\dot{H}^1}. 
\end{align}

Now suppose that $W$ is a vector field on $\mathcal{D}(S^1)$ defined in a neighborhood of $\eta$. 
Let $\eta(s,t,r)$ be a three-parameter family 
of diffeomorphisms such that $\eta(0,0,0)=\eta$ with 
$\frac{\partial}{\partial s}\eta(0,0,0)=V$, $\frac{\partial}{\partial t}\eta(s,0,0)= W_{\eta(s,0,0)}$ for all sufficiently small $s$, 
and $\frac{\partial}{\partial r}\eta(0,0,0)=Z$. It is clear that such a map $\eta(s,t,r)$ exists.
From the formulas \eqref{eq:Dconn}, \eqref{Dalpha} and \eqref{alphametric} we have 
\begin{align}  \nonumber 
\langle \nabla_V^{(\alpha)} &W, Z \rangle_{\alpha} 
= 
\frac{1}{4}\int_{S^1} \frac{(\nabla_V^{(\alpha)} W)_x Z_x}{\eta_x} \, dx 
	\\ \nonumber
& = 
-\tfrac{\partial}{\partial s}\big|_{s = 0} 
\tfrac{\partial}{\partial t}\big|_{t = 0} 
\tfrac{\partial}{\partial r}\big|_{r = 0} 
D^{(\alpha)} \big( \eta(s, t, 0) \| \eta(0, 0, r) \big) 
\\  \nonumber 
&= 
\frac{1}{1 - \alpha^2}
\tfrac{\partial}{\partial s}\big|_{s= 0} 
\tfrac{\partial}{\partial t}\big|_{t = 0} 
\tfrac{\partial}{\partial r}\big|_{r = 0} 
\int_{S^1} \eta_x(s,t,0)^{\frac{1-\alpha}{2}} \eta_x(0,0,r)^{\frac{1+\alpha}{2}} dx 
\\  \label{alphaconn} 
&= 
\frac{1}{4} \, 
\tfrac{\partial}{\partial s}\big|_{s = 0} 
\int_{S^1} W_x(\eta(s, 0,0)) \eta_x(s, 0,0)^{\frac{-1-\alpha}{2}} Z_x \eta_x^{\frac{-1+\alpha}{2}} dx 
\\  \nonumber 
&= 
\frac{1}{4} 
\int_{S^1} (DW \cdot V)_x Z_x \eta_x^{-1} dx 
- 
\frac{1+\alpha}{8} \int_{S^1} W_x V_x \eta_x^{\frac{-3-\alpha}{2}} Z_x \eta_x^{\frac{-1+\alpha}{2}} dx 
%\\   \nonumber 
%&= 
%\frac{1}{4} \int_{S^1} \frac{(DW \cdot V)_x Z_x}{\eta_x} dx 
%- 
%\frac{1+\alpha}{8} \int_{S^1} \frac{W_x V_x Z_x}{\eta_x^2} dx 
\\  \nonumber 
&= 
\frac{1}{4} \int_{S^1} 
\Big\{ ((DW \cdot V \big) \circ \eta^{-1} )_x 
- 
\frac{1+\alpha}{2} (V\circ\eta^{-1})_x (W\circ\eta^{-1})_x \Big\} 
(Z\circ\eta^{-1})_x \, dx  
\end{align}
and integrating by parts and using the fact that $Z$ is arbitrary we find that 
\begin{align*} 
(\nabla^{(\alpha)}_V W)_\eta 
= 
(DW \cdot V)(\eta) - \Gamma^{\alpha}_\eta (W, V) 
\end{align*} 
where the Christoffel map is given by \eqref{Gammaformula}. 

We will use the same notation for calculations in the remaining cases $\alpha= \pm 1$. 
From \eqref{H1met} and \eqref{Dalpha1} we have 
\begin{align} \nonumber 
\langle V, W \rangle_{-1} 
&= 
\langle V, W \rangle_{1} 
= 
-\tfrac{\partial}{\partial s}\big|_{s = 0} 
\tfrac{\partial}{\partial t}\big|_{t = 0} 
D^{(-1)}( \eta(s, 0) \| \eta(0, t) ) 
\\  \label{alphaPMmetric} 
&= 
-\frac{1}{4} 
\tfrac{\partial}{\partial s}\big|_{s = 0} 
\tfrac{\partial}{\partial t}\big|_{t = 0} 
\int_{S^1} \eta_x(s, 0) \big( \log \eta_x(s, 0) - \log \eta_x(0, t) \big) dx 
\\  \nonumber 
&= 
\frac{1}{4} 
\tfrac{\partial}{\partial s}\big|_{s = 0}
\int_{S^1} \eta_x(s, 0) \frac{W_x}{\eta_x} \, dx 
= 
\frac{1}{4} 
\int_{S^1} \frac{V_xW_x}{\eta_x} dx 
\\  \nonumber 
&= 
\langle V, W \rangle_{\dot{H}^1}. 
\end{align}

The corresponding affine connections can be obtained as in \eqref{alphaconn} using 
\eqref{eq:Dconn}, \eqref{Dalpha1} and \eqref{alphaPMmetric}. 
When $\alpha = -1$ we have 
\begin{align*} 
%\langle \nabla^{-1}_V W, Z \rangle_{-1} 
%&= 
\int_{S^1} \frac{ (\nabla^{(-1)}_V W)_x Z_x }{ \eta_x } \, dx 
&= 
-\tfrac{\partial}{\partial s}\big|_{s = 0} 
\tfrac{\partial}{\partial t}\big|_{t = 0} 
\tfrac{\partial}{\partial r}\big|_{r = 0} 
\int_{S^1} \eta_x (s, t, 0) \log\frac{\eta_x(s,t,0)}{\eta_x(0,0,r)} \, dx 
\\ 
&= 
\tfrac{\partial}{\partial s}\big|_{s = 0} 
\tfrac{\partial}{\partial t}\big|_{t = 0} 
\int_{S^1} \eta_x(s, t, 0) \frac{Z_x}{\eta_x} \, dx 
	\\
& = 
\tfrac{\partial}{\partial s}\big|_{s = 0} 
\int_{S^1} W_x(s, 0, 0) \frac{Z_x}{\eta_x} \,dx 
\\ 
&= 
\int_{S^1} ( DW \cdot V)_x \frac{Z_x}{\eta_x} \, dx 
\end{align*} 
from which we deduce that 
\begin{equation} \label{flat-1} 
\Gamma^{(-1)}_\eta (W, V) = 0. 
\end{equation} 
When $\alpha =1$ an analogous calculation gives 
\begin{equation} \label{flat1} 
\Gamma^{(1)}_\eta (W, V) 
= 
- A^{-1}\partial_x \left\{ (V \circ \eta^{-1})_x (W \circ \eta^{-1})_x ) \right\} \circ \eta. 
\end{equation} 
This establishes the first part of the theorem. 

For the second part we need to verify that for any vector fields $X$, $Y$ and $Z$ 
on $\mathcal{D}(S^1)/\Rot(S^1)$ we have 
\begin{align} \label{eq:duality} 
X \langle Y, Z \rangle_{\dot{H}^1} 
= 
\langle \nabla^{(\alpha)}_X Y, Z \rangle_{\dot{H}^1} + \langle Y, \nabla^{(-\alpha)}_X Z \rangle_{\dot{H}^1}. 
\end{align}
This can be done either by a direct calculation as above or else it can be deduced 
from general properties of divergences of the type \eqref{Dalpha} and \eqref{Dalpha1} 
which are discussed in Chapter 3 of \cite{AN}. 
The fact that $\nabla^{(0)}$ is the Levi-Civita connection of the $\dot{H}^1$ metric follows 
at once from \eqref{eq:duality}.

The equation for geodesics of $\nabla^{(\alpha)}$ on $\mathcal{D}(S^1)/\Rot(S^1)$ 
has the form 
\begin{align*} 
\frac{d^2\gamma}{dt^2} = \Gamma^{(\alpha)}_\gamma \Big( \frac{d\gamma}{dt}, \frac{d\gamma}{dt} \Big). 
\end{align*} 
Setting $d\gamma/dt = u \circ \gamma$ defines a time-dependent vector field $u$ on the circle $S^1$ 
(i.e., a periodic function vanishing at $x=0$). 
Differentiating this relation with respect to $t$ and eliminating the first and second derivatives of $\gamma$ 
from the geodesic equation gives 
$$
(u_t + uu_x)\circ\gamma 
= 
\Gamma^{(\alpha)}_\gamma(u \circ \gamma, u \circ \gamma). 
$$
Using \eqref{Gammaformula} and composing both sides with $\gamma^{-1}$ we obtain 
a nonlinear pseudodifferential equation 
$$
u_t + uu_x 
= 
- \frac{1 + \alpha}{2}A^{-1} \partial_x (u_x^2) 
$$
which we can rewrite as a nonlinear PDE 
$$
-u_{txx} - 3u_xu_{xx} - uu_{xxx} = - (1 + \alpha)u_xu_{xx} 
$$
yielding \eqref{PJ}. 
\end{proof} 

\begin{remark}\upshape
The Hunter-Saxton equation \eqref{HS} can be alternatively derived by observing that it is
is the Euler-Arnold equation of $\nabla^{(0)}$ on $T_e\mathcal{D}(S^1)/\Rot(S^1)$ 
and as such it is obtained from the geodesic equation of the right-invariant $\dot{H}^1$ metric 
\eqref{H1met} by a standard reduction procedure, see \cite{km}. 
%The derivation of \eqref{PJ} can be done similarly as follows. 
\end{remark}

\begin{remark} \upshape 
Another form of the Proudman-Johnson equation can be obtained by integrating \eqref{PJ} 
in the $x$ variable 
$$ 
u_{tx} + uu_{xx} + \frac{1 - \alpha}{2} \, u_x^2 = C(t) 
$$
where $C(t) = -\frac{1 + \alpha}{2} \int_{S^1} u_x^2 \, dx$. 
Observe that $C(t)$ is a conserved integral of the equation \eqref{HS} when $\alpha = 0$. 
%$$
%C_t(t) = \int u_x u_{tx} dx
%= \int u_x(-uu_{xx} - \frac{ 1 - \alpha}{2}u_x^2) dx
%= \frac{\alpha}{2} \int u_x^3 dx. 
%$$
\end{remark} 

\begin{remark}[$\alpha$-curvature] \upshape 
Using the Christoffel symbols \eqref{Gammaformula} it is possible to calculate the curvature of 
the $\alpha$-connections. It turns out to be proportional to the curvature of the $\dot{H}^1$ metric, 
i.e. for any vector fields $X$, $Y$ and $Z$ on $\mathcal{D}(S^1)/\Rot(S^1)$ we have 
\begin{equation} \label{Ralpha} 
R^{(\alpha)}(X,Y)Z 
= 
(1-\alpha^2) \Big( X \langle Y, Z \rangle_{\dot{H}^1} + Y\langle X,Z\rangle_{\dot{H}^1} \Big).
\end{equation} 
This formula can be computed as in finite dimensions; see \cite{mc} where a different choice of parameters 
is made. 
\end{remark} 

It turns out that the geodesic equation corresponding to \eqref{PJ} with $\alpha = 1$ 
can be integrated as well, albeit indirectly, by constructing affine coordinates for $\nabla^{(1)}$. 
Observe that from \eqref{Ralpha} we already know that the connections $\nabla^{(-1)}$ and $\nabla^{(1)}$ 
are flat. In the former case this is also evident from \eqref{flat-1}. 

\begin{theorem} \label{affineremark} 
The geodesic equation 
\begin{equation} \label{(1)} 
u_{txx} + u_xu_{xx} + uu_{xxx} = 0 
\end{equation} 
of $\nabla^{(1)}$ is integrable. Its general solution is given by
\begin{equation} \label{formulas} 
u = \frac{d\eta}{dt} \circ\eta^{-1} 
\quad 
\mbox{where} 
\quad 
\eta(t,x) = \frac{\int_0^x e^{a(y)t + b(y)} dy}{\int_{S^1} e^{a(x)t + b(x)} dx} 
\end{equation} 
and where $a, b$ are smooth mean-zero functions on $S^1$. 
\end{theorem} 
\begin{proof} 
We will construct a chart on $\mathcal{D}(S^1)/\Rot(S^1)$ in which the Christoffel symbols 
of $\nabla^{(1)}$ vanish identically. Consider the map
\begin{equation} \label{fi} 
\eta \mapsto \phi(\eta) = \log \eta_x - \int_{S^1} \log \eta_x \, dx 
\end{equation} 
from $\mathcal{D}(S^1)/\Rot(S^1)$ to the space of smooth periodic mean-zero functions. 
To see how the Christoffel symbols transform under $\eta \mapsto \tilde\eta = \phi(\eta)$ 
we first compute the derivatives 
$$
D_\eta\phi (W) 
= 
\frac{W_x}{\eta_x} - \int_{S^1} \frac{W_x}{\eta_x} \, dx, 
\qquad 
D^2_\eta \phi (W, V) = -\frac{V_x W_x}{\eta_x^2} + \int_{S^1} \frac{V_x W_x}{\eta_x^2} \, dx 
$$
for any $V, W \in T_\eta \mathcal{D}(S^1)/\Rot(S^1)$. 
Next, from \eqref{A} and \eqref{flat1} we obtain 
$$
\tilde{\Gamma}^{(1)}_{\phi(\eta)} \big( D_\eta\phi (W), D_\eta\phi (V) \big) 
=  
D^2_\eta\phi (W, V) + D_\eta\phi \big( \Gamma^{(1)}_\eta (W, V) \big) 
$$
\begin{align*} 
&= 
-\frac{ V_x W_x }{\eta_x^2} 
+ 
\int_{S^1} \frac{ V_x W_x }{\eta_x^2} \, dx 
- 
\frac{ ( A^{-1} (v_x w_x)_x \circ \eta )_x }{\eta_x}  
+ 
\int_{S^1} \frac{ (A^{-1} (v_x w_x)_x \circ \eta)_x }{\eta_x} \, dx 
\\ 
&= 
- (v_x w_x)\circ\eta 
+ 
\int_{S^1} (v_x w_x)\circ \eta \, dx 
- 
(A^{-1} (v_x w_x)_x)_x \circ \eta 
+ 
\int_{S^1} (A^{-1} (v_x w_x)_x)_x \circ \eta \, dx 
\\ 
&= 
- (v_x w_x)\circ\eta +  \int_{S^1} (v_x w_x)\circ \eta \, dx 
- 
\Big( - v_x w_x + \int_{S^1} v_x w_x \, dx \Big) \circ \eta 
\\ 
&\qquad\qquad\qquad\qquad\qquad\qquad\qquad\qquad 
+
\int_{S^1}  \Big( - v_x w_x + \int_{S^1} v_x w_x dx \Big) \circ \eta \, dx 
= 0,
\end{align*}
where $v = V\circ\eta^{-1}$ and $w = W\circ\eta^{-1}$. 

We can now solve \eqref{(1)} as follows. 
Since $\tilde\Gamma^{(1)} \equiv 0$ all geodesics of $\nabla^{(1)}$ in the affine coordinates 
are the straight lines which can be written as 
$$ 
t \to \tilde{\eta}(t,x) = a(x) t + b(x) 
\qquad\qquad\qquad 
x \in S^1 
$$ 
for some smooth mean-zero periodic functions $a$ and $b$. 
Thus, given any such functions to construct a general solution $u$ it is sufficient to 
(i) invert the map $\phi$ in \eqref{fi} to obtain the flow $t \to \eta(t) = \phi^{-1}\tilde{\eta}(t)$ 
and 
(ii) right-translate the velocity vector of $\eta(t)$ to the tangent space at the identity 
in $\mathcal{D}(S^1)/\Rot(S^1)$.\footnote{In the language of fluid dynamics the second step 
corresponds to going from Lagrangian to Eulerian coordinates.} 
The required formulas are those in \eqref{formulas}. 
\end{proof} 

It is worth pointing out that the above proof manifests integrability of (\ref{(1)}) in that 
it provides an explicit change of coordinates that linearizes the flow in the same spirit as 
the inverse scattering transform formalism.

%%%%%%%%%%%%%%%%%%%%%%%%%%%
\section{The $n$-dimensional case: $\mathcal{D}(M)/\mathcal{D}_\mu(M)$} 

We now turn to the general case when $M$ is an $n$-dimensional compact Riemannian manifold 
without boundary and work with the coset space $\mathcal{D}(M)/\mathcal{D}_\mu(M)$. 
Our next result is stated in analogy with Theorem \ref{thm1}. 

\begin{theorem} \label{thm3} \mbox{} 
\begin{enumerate} 
\item[(i)] 
Each divergence $D^{(\alpha)}$ induces on the quotient $\mathcal{D}(M)/\mathcal{D}_\mu(M)$ 
the $\dot{H}^1$ metric \eqref{H1met} and an affine connection $\nabla^{(\alpha)}$ 
given for a right-invariant vector field $W_\eta = w \circ \eta$ 
and a tangent vector $V = v \circ \eta$ by
\begin{align}\label{ndimnabla}
 ( \nabla^{(\alpha)}_V W)_\eta 
 = 
 -\left\{ \Delta^{-1}d \Bigl( 
 d\diver w \cdot  v +  \frac{1-\alpha}{2}  \diver{w} \diver{v} 
 \Bigr) \right\}^\sharp \circ \eta 
\end{align}
where $\alpha \in [-1, 1]$ and $\Delta = d\delta + \delta d$ denotes the Laplace-de Rham operator.
\item[(ii)] 
For any $\alpha$ the connections $\nabla^{(\alpha)}$ and $\nabla^{(-\alpha)}$ are dual 
with respect to the $\dot{H}^1$ metric and $\nabla^{(0)}$ is the Levi-Civita connection. 
\item[(iii)] 
The geodesic equation of $\nabla^{(\alpha)}$ on $\mathcal{D}(M)/\mathcal{D}_\mu(M)$ 
is equivalent to the following nonlinear PDE 
\begin{align} \label{PJn} 
%d \mathrm{div} (u_t) + d \iota_u d \mathrm{div}(u) + (1-\alpha) \mathrm{div}(u) \, d\mathrm{div}(u) = 0. 
d\varphi_t + d\iota_{\displaystyle u} d\varphi + (1-\alpha) \varphi \, d\varphi = 0, 
\qquad 
\varphi = \mathrm{div}\, u. 
\end{align} 
\end{enumerate} 
\end{theorem} 
\begin{proof} 
Despite the fact that we are now working with cosets the computations involved in the proof of 
statements (i) and (ii) are similar to the one-dimensional case in Section \ref{section2}. 
We will prove (i) for $\alpha \in (-1,1)$; the proofs for $\alpha = \pm 1$ are analogous 
and will be omitted. 

Let $W$ be a right-invariant vector field on $\mathcal{D}(M)$ defined in a neighborhood of 
$\eta \in \mathcal{D}(M)$. 
Given $V, Z \in T_\eta \mathcal{D}(M)$ let $\eta(s,t,r)$ be a three-parameter family of 
diffeomorphisms such that $\eta(0,0,0)=\eta$ with 
$\frac{\partial}{\partial s}\eta(0,0,0)=V$, $\frac{\partial}{\partial t}\eta(s,0,0)= W_{\eta(s,0,0)}$ 
for all sufficiently small $s$ and $\frac{\partial}{\partial r}\eta(0,0,0)=Z$. 
%The identity
%% 
%\begin{align}\label{DJac} 
%D_\eta \Jac_\mu(V) 
%= 
%\diver(V \circ \eta^{-1}) \circ \eta \,\Jac_\mu\eta, 
%\qquad 
%V \in T_\eta \mathcal{D}(M) 
%\end{align}  
%% 
%and the right invariance of $W$ imply 
% 
Using the identity 
\begin{align*}
\tfrac{\partial}{\partial t}\big|_{t = 0} \Jac_\mu\eta(s,t,0)
& = \diver(W_{\eta(s,0,0)} \circ \eta^{-1}(s,0,0)) \circ \eta(s,0,0) \, \Jac_\mu\eta(s,0,0)
	\\
& = \diver w \circ \eta(s,0,0) \, \Jac_\mu\eta(s,0,0),
\end{align*}
where $W_\eta = w \circ \eta$ 
and similar computations for the partial derivatives in $s$ and $r$ 
we have 
\begin{align}  \nonumber 
\langle & \nabla_V^{(\alpha)} W, Z \rangle_{\alpha} 
= 
\frac{1}{4}\int_{M} \diver\bigl((\nabla_V^{(\alpha)} W) \circ \eta^{-1}\bigr) \diver(Z \circ \eta^{-1}) \, d\mu 
	\\ \nonumber
&=  
-\tfrac{\partial}{\partial s}\big|_{s = 0} 
\tfrac{\partial}{\partial t}\big|_{t = 0} 
\tfrac{\partial}{\partial r}\big|_{r = 0} 
D^{(\alpha)} \big( \eta(s, t, 0) \| \eta(0, 0, r) \big) 
\\  \nonumber 
&= 
\frac{1}{1 - \alpha^2}
\tfrac{\partial}{\partial s}\big|_{s= 0} 
\tfrac{\partial}{\partial t}\big|_{t = 0} 
\tfrac{\partial}{\partial r}\big|_{r = 0} 
\int_{M} (\Jac_\mu\eta(s,t,0))^{\frac{1-\alpha}{2}} (\Jac_\mu \eta(0,0,r))^{\frac{1+\alpha}{2}} d\mu 
\\  \nonumber
&= 
\frac{1}{4} \, 
\tfrac{\partial}{\partial s}\big|_{s = 0} 
\int_{M} \diver w \circ \eta(s,0,0) \, (\Jac_\mu \eta(s,0,0))^{\frac{1-\alpha}{2}} 
\diver(Z \circ \eta^{-1}) \circ\eta \, (\Jac_\mu \eta)^{\frac{1+\alpha}{2}} d\mu 
\\  \nonumber 
&= 
\frac{1}{4} 
\int_M \Bigl\{ d\diver w \circ \eta \cdot V 
	\\ \nonumber
&\hspace{2cm}+ \frac{1-\alpha}{2} \diver w \circ\eta 
\diver(V\circ\eta^{-1})\circ\eta\Bigr\} \diver(Z\circ\eta^{-1})\circ\eta \, \Jac_\mu\eta \, d\mu
\\  \nonumber 
&= 
\frac{1}{4} 
\int_M \Bigl\{ d\diver w \cdot v
+ 
\frac{1-\alpha}{2} \diver{w} \diver{v}\Bigr\} \diver(Z\circ\eta^{-1}) \, d\mu,
\end{align}
where $V = v \circ \eta$. 
Using integration by parts 
$$
\int_M f \diver X \, d\mu = - \int_M f \delta X^\flat d\mu = - \int_M df (X) \, d\mu
$$
and the fact that $Z$ is arbitrary, we find  
\begin{align*} 
d\diver\bigl\{(\nabla_V^{(\alpha)} W) \circ \eta^{-1}\bigr\}
= 
df \quad \text{i.e.} \quad \diver\bigl\{(\nabla_V^{(\alpha)} W) \circ \eta^{-1}\bigr\} 
= 
f - \int_M f d\mu,
\end{align*} 
where $f = d\diver w \cdot v + \frac{1-\alpha}{2} \diver{w} \diver{v}$.
The expression in (\ref{ndimnabla}) now follows from the next lemma. 
\begin{lemma}
  Let $\Delta$ be the Laplace-de Rham operator. Then
  $$\diver (-(\Delta^{-1} df)^\sharp) = f - \int_M f d\mu,  \qquad f \in C^\infty(M).$$
\end{lemma}
\begin{proof} 
If $g = \diver\bigl(-(\Delta^{-1} df)^\sharp\bigr) = \delta \Delta^{-1} df$ then $\Delta g = \delta df = \Delta f$ 
and 
since the kernel of $\Delta$ acting on functions consists of the constants, we deduce that 
$g = f + c$ for some constant $c$. Integrating over $M$ yields $0 = \int_M f d\mu + c$,
which determines $c$.
\end{proof} 

Calculation of $\nabla_V^{(\alpha)} W$ for an arbitrary (not necessarily right-invariant) vector field $W$ can be reduced to the right-invariant case as follows. Let $\eta \in \mathcal{D}(M)$ and let $W^R$ be the right-invariant vector field with $W^R_\eta = W_\eta$. 
Then 
\begin{align} \label{generalnabladef}
( \nabla_V^{(\alpha)} W )_\eta 
= 
(\nabla_V^{(\alpha)} (W - W^R))_\eta + ( \nabla_V^{(\alpha)} W^R )_\eta 
= 
(\mathcal{L}_V(W - W^R))_\eta + ( \nabla_V^{(\alpha)} W^R )_\eta,
\end{align}
where $\mathcal{L}_V$ denotes the Lie derivative in the direction of any vector field $\tilde{V}$ such that $\tilde{V}_\eta = V$. Indeed, if $Z$ is a vector fields satisfying $Z_\eta = 0$, then locally
$$(\nabla_V Z)_\eta = DZ(\eta) \cdot V - \Gamma_\eta(Z_\eta, V) = DZ(\eta) \cdot V = [V, Z]_\eta = (\mathcal{L}_V Z)_\eta.$$
Using (\ref{ndimnabla}) and (\ref{generalnabladef}) we can compute the affine connection 
$\nabla^{(\alpha)}$ for arbitrary fields on $\mathcal{D}(M)$ and this completes the proof of (i).

As in the proof of Theorem \ref{thm1}, (ii) can be established by a direct calculation 
or, alternatively, it can be deduced from the general properties of divergences of 
the type \eqref{Dalpha} and \eqref{Dalpha1}.

Regarding (iii) 
let $\gamma(t)$ be a curve in $\mathcal{D}(M)$ with $\gamma(0) = e$ and 
$\dot{\gamma}= d\gamma/dt = u \circ \gamma$. 
In the appendix we prove that \eqref{generalnabladef} implies 
\begin{align}\label{nablarelation}
\nabla_{\dot{\gamma}}^{(\alpha)}\dot{\gamma}
= 
u_t \circ \gamma + \nabla_{\dot{\gamma}}^{(\alpha)} \dot{\gamma}^R.
\end{align}
Using (\ref{ndimnabla}) and (\ref{nablarelation}) the geodesic equation 
$\nabla_{\dot{\gamma}}^{(\alpha)} \dot{\gamma} = 0$ 
can be written as 
$$
u_t 
= 
(\Delta^{-1}df)^\sharp 
\quad 
\mbox{where} 
\quad 
f = d\diver u \cdot u +  \frac{1-\alpha}{2}  (\diver{u})^2 
$$
so that setting $\varphi = \mathrm{div}\, u$ and taking the exterior derivative we get 
$$
-d\varphi_t + \delta d u_t^\flat 
= 
d\iota_ud\varphi + (1-\alpha)\varphi d\varphi. 
$$
The relation $u_t^\flat = \Delta^{-1} df$ implies $\Delta d u_t^\flat = d \Delta u_t^\flat = 0$ 
hence, in particular, $\delta du_t^\flat = 0$. 
This proves (\ref{PJn}).
\end{proof} 

Finally, we turn to integrability of the geodesic equations in \eqref{PJn} which to the best of our knowledge 
has not been studied in the literature before, with the exception of the case $\alpha =0$ in \cite{klmp}. 

\begin{corollary} \label{affineremarkndim} 
The geodesic equations of $\nabla^{(\alpha)}$ on $\mathcal{D}(M)/\mathcal{D}_\mu(M)$ 
corresponding to \eqref{PJn} with $\alpha =0, \pm 1$ are integrable equations in any dimension $n$. 
\end{corollary} 
\begin{proof} 
Equation (\ref{PJn}) with $\alpha =0$ was derived and shown to be integrable in \cite{klmp}. 
Since $\nabla^{(1)}$ and $\nabla^{(-1)}$ are flat, integrability of the other two equations 
can be established similarly to their one-dimensional analogues in Section \ref{section2}. 
We will consider the more complicated case $\alpha = 1$ with the corresponding equation 
\begin{equation} \label{(1ndim)} 
d\varphi_t + d\iota_u d\varphi = 0, 
\qquad 
\varphi = \diver{u}. 
\end{equation} 
Proceeding as in the proof of Theorem \ref{affineremark} we first show that the map 
$$ 
\eta \mapsto \phi(\eta) = \log \Jac_\mu\eta - \int_{M} \log \Jac_\mu \eta \, d\mu 
$$  
from $\mathcal{D}(M)/\mathcal{D}_\mu(M)$ to the space of smooth mean-zero functions on $M$ 
defines an affine chart for $\nabla^{(1)}$. 
Thus, as before, the geodesics of $\nabla^{(1)}$ in the affine coordinates defined by $\phi$ 
are the straight lines
$$ 
t \to \tilde{\eta}(t,x) = a(x) t + b(x) 
\qquad\qquad\quad 
x \in M 
$$ 
where $a$ and $b$ are smooth mean-zero functions on $M$. 
From this we find 
$$ 
\Jac_\mu \eta(t,x) = \frac{e^{a(x) t + b(x)}}{\int_M e^{a(x) t + b(x)} d\mu} 
\qquad\qquad 
x \in M 
$$ 
and combining this expression with the identity 
$ 
\frac{d}{dt}\Jac_\mu(\eta) = (\varphi \circ \eta) \Jac_\mu\eta, 
$ 
we obtain 
\begin{equation} \label{eq:sol} 
\varphi(t, \eta(t,x))
= 
a(x) - \frac{\int_M a(x) e^{a(x) t + b(x)} d\mu}{\int_M e^{a(x) t + b(x)} d\mu}.
\end{equation} 
Note that the time derivative of $\varphi \circ \eta$ is independent of $x$, 
so that 
$$ 
d\varphi_t + d\iota_u d\varphi
= 
d\bigl((\varphi \circ \eta)_t \circ \eta^{-1}\bigr)
= 0 
$$ 
which shows that \eqref{eq:sol} solves the equations in \eqref{(1ndim)}. 
\end{proof} 

%%%%%%%%%%%%%%%%%%%%%%%%%%%%%%%
\appendix
\section{Proof of (\ref{nablarelation})}\label{app}
\renewcommand{\theequation}{A.\arabic{equation}}
Let $\gamma(t)$ be a curve in $\mathcal{D}(M)$ with $\dot{\gamma}(t) = u(t) \circ \gamma(t)$. For any $t_0 >0$ we compute
\begin{align}\nonumber 
u_t(t_0) \circ \gamma(t_0) & = \biggl\{\frac{d}{dt}\bigg|_{t = t_0} \left(\dot{\gamma}(t) \circ \gamma(t)^{-1}\right)\biggr\} \circ \gamma(t_0)
	\\\label{A1}
& = \ddot{\gamma}(t_0) + D(\dot{\gamma}(t_0)) \cdot \biggl\{\frac{d}{dt}\bigg|_{t=t_0} \gamma(t)^{-1}\biggr\}\circ \gamma(t_0).
\end{align}
But differentiating the relation $\gamma(t)^{-1} \circ \gamma(t) = id$ with respect to $t$ we find
$$\biggl\{\frac{d}{dt}\bigg|_{t=t_0} \gamma(t)^{-1}\biggr\}\circ \gamma(t_0)
+ D(\gamma(t_0)^{-1})\circ\gamma(t_0) \cdot \dot{\gamma}(t_0) = 0.$$
Thus, (\ref{A1}) yields
\begin{align*}
u_t(t_0) \circ \gamma(t_0) 
& = \ddot{\gamma}(t_0) - D(\dot{\gamma}(t_0)) \cdot D(\gamma(t_0)^{-1})\circ\gamma(t_0) \cdot \dot{\gamma}(t_0)
	\\
& = \frac{d}{dt}\bigg|_{t = t_0} \left\{\dot{\gamma}(t) - \dot{\gamma}(t_0) \circ \gamma(t_0)^{-1} \circ \gamma(t)\right\}
= D(\dot{\gamma} - \dot{\gamma}^R) \cdot \dot{\gamma}(t_0).
\end{align*}
Equation (\ref{nablarelation}) now follows from (\ref{generalnabladef}).

%The geodesic equation associated with a right-invariant connection on a Lie group $G$ is $u_t = -B(u,u)$, where the bilinear map $B:\mathfrak{g} \times \mathfrak{g} \to \mathfrak{g}$ is defined by
%$$B(v,w) = (\nabla_V W)_e - \frac{1}{2}[V, W]_e$$
%and $V,W$ are the unique right-invariant vector fields such that $V_e = v$ and $W_e = w$ \cite{A1966}. 
%Thus, equation (\ref{ndimnabla}) leads to the geodesic equation
%$$u_t  = (\Delta^{-1}df)^\sharp,\qquad f = d\diver(u) \cdot  u +  \frac{1-\alpha}{2}  \diver{u} \diver{u},$$

\bibliographystyle{plain}
\bibliography{is}

\begin{thebibliography}{9999}
\small
\bibitem{AN}
S.-I. Amari and H. Nagaoka, 
\emph{Methods of Information Geometry}, 
American Mathematical Society, Providence, RI 2000. 

\bibitem{A1966}
V. I. Arnold,
Sur la g\'eometrie diff\'erentielle des groupes de Lie de dimension infinie et
ses application \`a l'hydrodynamique des fluides parfaits,
{\it Ann. Inst. Fourier (Grenoble)} {\bf 16} (1966), 319--361.


\bibitem{ch} 
N. N. Chentsov, 
\textit{Statistical Decision Rules and Optimal Inference}, 
American Mathematical Society, Providence, RI 1982. 

\bibitem{kv} 
R. Kass and P. Vos, 
\textit{Geometrical Foundations of Asymptotic Inference}, 
Wiley-Interscience, New York 1997. 

\bibitem{klmp} 
B. Khesin, J. Lenells, G. Misio\l ek and S. Preston, 
\textit{Geometry of diffeomorphism groups, complete integrability and geometric statistics}, 
to appear in Geom. funct. anal. (2012). 

\bibitem{km} 
B. Khesin and G. Misio\l ek, 
\textit{Euler equations on homogeneous spaces and Virasoro orbits}, 
Adv. Math. \textbf{176} (2003), 116--144. 

\bibitem{le} 
J. Lenells, 
\textit{The Hunter-Saxton equation describes the geodesic flow on a sphere}, 
J. Geom. Phys. \textbf{57} (2007), 2049--2064. 

\bibitem{mc} 
E. A. Morozova and N. N. Chentsov, 
\textit{Natural geometry of families of probability laws} (in Russian), 
Itogi Nauki i Tekhniki, Akad. Nauk SSSR, \textbf{83} (1991), 133--265. 

\bibitem{v} 
C. Villani, 
\textit{Optimal Transport: Old and New}, 
Springer, New York 2009. 

\end{thebibliography}

\end{document}